\documentclass[11pt]{article}
\usepackage[a4paper,margin=1.0in]{geometry}
\usepackage[colorlinks,citecolor=magenta,linkcolor=black]{hyperref}
\pdfpagewidth=\paperwidth \pdfpageheight=\paperheight
\usepackage{amsfonts,amssymb,amsthm,amsmath,eucal,tabu,url}
\usepackage{pgf}
\usepackage{array}
\usepackage{tikz-cd}
\usepackage{pstricks}
\usepackage{pstricks-add}
\usepackage{pgf,tikz}
\usetikzlibrary{automata}
\usetikzlibrary{arrows}
\usepackage{indentfirst}
\usepackage{tabularx} 
\theoremstyle{plain}
\newtheorem{thm}{Theorem}[section]
\newtheorem{theorem}[thm]{Theorem}

\newtheorem{lemma}[thm]{Lemma}

\newtheorem{corollary}[thm]{Corollary}

\theoremstyle{definition}
\newtheorem{definition}[thm]{Definition}
\newtheorem{remark}[thm]{Remark}
\newtheorem{example}[thm]{Example}

\newtheorem{thevarthm}[thm]{\varthmname}

\newenvironment{varthm*}[1]{\trivlist\item[]{\bf #1.}\it}{\endtrivlist}
\def\keywordname{{\bfseries Keywords}}%
\def\keywords#1{\par\addvspace\medskipamount{\rightskip=0pt plus1cm
\def\and{\ifhmode\unskip\nobreak\fi\ $\cdot$
}\noindent\keywordname\enspace\ignorespaces#1\par}}
\def\subclassname{{\bfseries Mathematics Subject Classification
(2020)}\enspace}
\def\subclass#1{\par\addvspace\medskipamount{\rightskip=0pt plus1cm
\def\and{\ifhmode\unskip\nobreak\fi\ $\cdot$
}\noindent\subclassname\ignorespaces#1\par}}

\begin{document}
\title{On the addition technique for Betti and Poincar\'e polynomials of plane curves}
\author{Piotr Pokora}
\date{\today}
\maketitle

\thispagestyle{empty}
\begin{abstract}
Using the addition technique, we present polynomial identities for the Betti and Poincar\'e polynomials of reduced plane curves.
\keywords{reduced curves, quasi-homogeneous singularities, conic-line arrangements}
\subclass{14H50, 32S25, 14N10}
\end{abstract}
\section{Introduction}
In this paper we study Poincar\'e polynomials for plane curves with quasi-homogeneous singularities in the context of the addition technique. Our study is motivated by a classical result from the theory of central hyperplane arrangements in $\mathbb{K}^{n}$, where $\mathbb{K}$ is an arbitrary field. We present here a brief historical outline. Let $V = \mathbb{K}^{n}$ be a vector space, a hyperplane $H$ in $V$ is a linear subspace of dimension $n-1$. An arrangement of hyperplanes $\mathcal{A}$ is a finite set of hyperplanes in $V$. We denote by $L = L(\mathcal{A})$ the set of intersections of hyperplanes in $\mathcal{A}$ that is partially ordered by the reverse inclusion. We define a rank function for the elements in $L(\mathcal{A})$, namely for $X \in L(\mathcal{A})$ one has
$$r(X) = n - {\rm dim} \, X.$$
Let $\mu : L(\mathcal{A}) \times L(\mathcal{A}) \rightarrow \mathbb{Z}$ be the M\"obius function of $L$ and for $X\in L(\mathcal{A})$  we define $\mu(X) : = \mu(V,X)$. A Poincar\'e polynomial of $\mathcal{A}$ is defined by
$$\pi(\mathcal{A};t) = \sum_{X \in L(\mathcal{A})}\mu(X)(-t)^{r(X)}.$$
It is well-known that the Poincar\'e polynomial is a degree $r(\mathcal{A})= \max_{X\in L(\mathcal{A})} r(X)$ polynomial in $t$ with non-negative coefficients. For a hyperplane $H_{0} \in \mathcal{A}$ we define a triple of the form $(\mathcal{A}, \mathcal{A}', \mathcal{A}'')$, where $A' = \mathcal{A} \setminus \{H_{0}\}$ and $\mathcal{A}'' = \{H \cap H_{0} \neq \emptyset \, : \, H \neq H_{0} \text{ and } H \in \mathcal{A}'\}.$
\begin{theorem}
For a triple $(\mathcal{A},\mathcal{A}', \mathcal{A}'')$ one has the following relation:
$$\pi(\mathcal{A};t) = \pi(\mathcal{A}';t) + t \pi(\mathcal{A}'';t).$$
\end{theorem}
Moreover, if we restrict our attention to the case $\mathbb{K} = \mathbb{C}$, then the Poincar\'e polynomial allows us to understand the topology of the complement of $\mathcal{A}$ in the projectivized situation. Motivated by the above classical result, we would like to understand properties of Poincar\'e polynomials of plane curves under the so-called addition technique. The concept of the Poincar\'e polynomial of a reduced plane curve $C$ of degree $d$ has been very recently introduced in \cite{Dimca1, Pokora}, and this polynomial is defined as follows
$$\mathfrak{P}(C;t) = 1 + (d-1)t + ((d-1)^{2} - \tau(C))t^{2},$$
where $\tau(C)$ denotes the total Tjurina number of $C$. It turned out that this polynomial decodes the freeness property and it allows to compute the Euler number of the complement 
$$M(C) := \mathbb{P}^{2}_{\mathbb{C}} \setminus C.$$ It is natural to wonder whether we can say something about the properties of Poincar\'e polynomials once we add curves. The main result of the paper is the following general result that holds under the assumption that our curves have quasi-homogeneous singularities.
\begin{theorem}
Let $C_{1},C_{2} \subset \mathbb{P}^{2}_{\mathbb{C}}$ be two reduced curves such that $C_{1} \cup C_{2}$ admits only quasi-homogeneous singularities. Assume that $C_{1} \cap C_{2}$ is $0$-dimensional consisting of $r$ points, then one has
$$\mathfrak{P}(C_{1} \cup C_{2};t) = \mathfrak{P}(C_{1};t) + \mathfrak{P}(C_{2};t) + t-1 + (r-1)t^{2}.$$
\end{theorem}
The second result tells us how the Betti polynomial behaves under the addition technique. Let us recall that the Betti polynomial of $M(C)$ is defined as
$$B_{M(C)}(t) = 1 + (e-1)t + ((d-1)^{2} - \mu(C) - d+e)t^{2},$$
see Section $3$ for further explanations.
\begin{theorem}
Let $C_{1},C_{2} \subset \mathbb{P}^{2}_{\mathbb{C}}$ be two reduced curves such that $C_{1} \cap C_{2}$ is $0$-dimensional consisting of $r$ points. Then one has
    $$B_{M(C_{1} \cup C_{2})}(t) = B_{M(C_{1})}(t) + B_{M(C_{2})}(t) +t-1+(r-1)t^{2}.$$
\end{theorem}
Let us give an outline of our paper. In Section $2$, we recall some basics about free plane curves and their Poincar\'e polynomials. In Section $3$, we give our proofs of the above two results. We work exclusively over the complex numbers.
\section{Preliminaries}
We follow the notation introduced in \cite{Dimca}. Let us denote by $S := \mathbb{C}[x,y,z]$ the coordinate ring of $\mathbb{P}^{2}_{\mathbb{C}}$. For a homogeneous polynomial $f \in S$ let us denote by $J_{f}$ the Jacobian ideal associated with $f$, i.e., the ideal $J_{f} = \langle \partial_{x}\, f, \partial_{y} \, f, \partial_{z} \, f \rangle$.
\begin{definition}
Let $p$ be an isolated singularity of a polynomial $f\in \mathbb{C}[x,y]$. Since we can change the local coordinates, assume that $p=(0,0)$.
\begin{itemize}
    
\item The number 
$$\mu_{p}=\dim_\mathbb{C}\left(\mathbb{C}\{x,y\} /\bigg\langle \partial_{x}\, f ,\partial_{y}\, f \bigg\rangle\right)$$
is called the Milnor number of $f$ at $p$.

\item The number
$$\tau_{p}=\dim_\mathbb{C}\left(\mathbb{C}\{x,y\}/\bigg\langle f, \partial_{x}\, f ,\partial_{y} \, f \bigg\rangle \right)$$
is called the Tjurina number of $f$ at $p$.

\end{itemize}
\end{definition}

The total Tjurina number of a given reduced curve $C \subset \mathbb{P}^{2}_{\mathbb{C}}$ is defined as
$${\rm deg}(J_{f}) = \tau(C) = \sum_{p \in {\rm Sing}(C)} \tau_{p}.$$ 

Recall that a singularity is called quasi-homogeneous if and only if there exists a holomorphic change of variables so that the defining equation becomes weighted homogeneous. If $C : f=0$ is a reduced plane curve with only quasi-homogeneous singularities, then one has $\tau_{p}=\mu_{p}$ for all $p \in {\rm Sing}(C)$, and eventually 
$$\tau(C) = \sum_{p \in {\rm Sing}(C)} \mu_{p} = \mu(C),$$
which means that the total Tjurina number is equal to the total Milnor number of $C$.

Next, we will need an important invariant that is defined using the syzygies associated with the Jacobian ideal $J_{f}$.
\begin{definition}
Consider the graded $S$-module of Jacobian syzygies of $f$, namely $$AR(f)=\{(a,b,c)\in S^3 : a\partial_{x} \, f + b \partial_{y} \, f + c \partial_{z} \, f = 0 \}.$$
The minimal degree of non-trivial Jacobian relations for $f$ is defined to be 
$${\rm mdr}(f):=\min_{r\geq 0}\{AR(f)_r\neq 0\}.$$ 
\end{definition}
\begin{remark}
If $C: f=0$ is a reduced plane curve in $\mathbb{P}^{2}_{\mathbb{C}}$, then we write ${\rm mdr}(f)$ or ${\rm mdr}(C)$ interchangeably.
\end{remark}
Let us now formally define the freeness of a reduced plane curve that was formally introduced in \cite{KS1}.
\begin{definition}
A reduced curve $C \subset \mathbb{P}^{2}_{\mathbb{C}}$ of degree $d$ is free if the Jacobian ideal $J_{f}$ is saturated with respect to $\mathfrak{m} = \langle x,y,z\rangle$. Moreover, if $C$ is free, then the pair $(d_{1}, d_{2}) = ({\rm mdr}(f), d - 1 - {\rm mdr}(f))$ is called the exponents of $C$.
\end{definition}
It is notoriously difficult to check the freeness property according to the above definition. However, we can use the following result, which provides an effective criterion \cite{duP}.
\begin{theorem}[du Plessis--Wall]
\label{dddp}
Let $C : f=0$ be a reduced curve in $\mathbb{P}^{2}_{\mathbb{C}}$. One has
\begin{equation*}
\label{duPles}
(d-1)^{2} - d_{1}(d-d_{1}-1) = \tau(C)
\end{equation*}
if and only if $C : f=0 $ is a free curve, and then ${\rm mdr}(f) = d_{1} \leq (d-1)/2$.
\end{theorem}
From the above result we have the following important corollary.
\begin{corollary}
If $C : f=0$ be a reduced free curve in $\mathbb{P}^{2}_{\mathbb{C}}$ of degree $d$ with exponents $(d_{1},d_{2})$, then
$$d_{1}d_{2} = (d-1)^{2} - \tau(C).$$
\end{corollary}
Let us now define the main object of our considerations, namely the Poincar\'e polynomial.
\begin{definition}
Let $C : f=0$ be a reduced curve in $\mathbb{P}^{2}_{\mathbb{C}}$. Then its Poincar\'e polynomial is defined as
$$\mathfrak{P}(C;t) = 1 + (d-1)t + ((d-1)^{2} - \tau(C))t^{2}.$$
\end{definition}

\begin{theorem}
If $C : f=0$ be a reduced free curve in $\mathbb{P}^{2}_{\mathbb{C}}$ of degree $d$ with exponents $(d_{1},d_{2})$, then its Poincar\'e polynomial splits over the rationals as
$$\mathfrak{P}(C;t) = (1+d_{1}t)(1+d_{2}t).$$
\end{theorem}
\begin{proof}
Recall that the freeness of $C$ implies that $d_{1}+d_{2} = d-1$ and $d_{1}d_{2} = (d-1)^{2} - 
\tau(C)$, hence
$$\mathfrak{P}(C;t) = 1 + (d-1)t + ((d-1)^{2} - \tau(C))t^{2} = 1 + (d_{1}+d_{2})t + d_{1}d_{2}t^{2} = (1+d_{1}t)(1+d_{2}t),$$
which completes the proof.
\end{proof}

\section{Proofs of the main results}
In order to show our main result we need a tiny preparation.
For a reduced plane curve $C$ of degree $d\geq 3$ its complement is defined as $M(C) = \mathbb{P}^{2}_{\mathbb{C}} \setminus C$. Recall that the Betti polynomial of $M(C)$ has the following presentation
\begin{equation*}
B_{M(C)}(t) = 1+b_{1}(M(C))t + b_{2}(M(C))t^{2} =  1 + (e-1)t + ((d-1)^{2} - \mu(C) - d+e)t^{2},  
\end{equation*}
where $\mu(C)$ denotes the total Milnor number of $C$ and $e$ is the number of irreducible components of $C$, see \cite[Theorem 3.1]{Dimca1} for all necessary details.

Let $C_{i} \subset \mathbb{P}^{2}_{\mathbb{C}}$ with $i=1,2$ be two reduced curves such that ${\rm deg}\, C_{i} = c_{i}$ and $C_{i}$ has exactly $e_{i}\geq 1$ irreducible components. We assume that $C_{1} \cap C_{2}$ is $0$-dimensional and consists of $r < \infty$ points. We use the notation $\mu(C,p)$ meaning the Milnor number of $C$ at its singular point $p \in C$. We will need the following crucial lemma that explains the behaviour of the total Milnor number under the addition of two curves.
\begin{lemma}
In the setting as above, one has
$$\mu(C_{1}\cup C_{2}) = \mu(C_{1}) + \mu(C_{2}) + 2c_{1}c_{2}-r.$$
\end{lemma}
\begin{proof}
Observe that for each intersection point $p \in C_{1}\cap C_{2}$ one has
$$\mu(C_{1} \cup C_{2},p) = \mu(C_{1},p) + \mu(C_{2},p) + 2i_{p}(C_{1},C_{2}) - 1,$$
where $i_{p}(C_{1},C_{2})$ denotes the intersection index of curves $C_{1}, C_{2}$ at $p$, see \cite[Theorem 6.5.1]{Wall}. Then the claim follows by summing the above relation over all singular points $p \in C_{1}\cap C_{2}$ and adding Milnor numbers of singular points on $C_{1} \setminus C_{1} \cap C_{2}$ and $C_{2}\setminus C_{1} \cap C_{2}$.
\end{proof}
Now we are ready to present our main result.
\begin{theorem}
In the setting as above, if all singularities of a reduced curve $C_{1} \cup C_{2}\subset \mathbb{P}^{2}_{\mathbb{C}}$ are quasi-homogeneous, then one has
\begin{equation}
\label{addition}
\mathfrak{P}(C_{1} \cup C_{2};t) = \mathfrak{P}(C_{1};t) + \mathfrak{P}(C_{2};t) + t-1 + (r-1)t^{2}.
\end{equation}
\end{theorem}
\begin{proof}
Since our curve $C_{1} \cup C_{2}$ has only quasi-homogeneous singularities we have $\tau(C_{1} \cup C_{2}) = \mu(C_{1} \cup C_{2})$ and $\mu(C_{i}) = \tau(C_{i})$ for $i = 1,2$, see \cite[Remark 2.4]{POG}. We start by computations preformed on the left-hand side, we have
\begin{align*}
\mathfrak{P}(C_{1} \cup C_{2};t) = 1 + (c_{1}+c_{2}-1)t + \bigg((c_{1}+c_{2}-1)^{2} - \mu(C_{1}\cup C_{2})\bigg)t^{2} &= \\
1+(c_{1}+c_{2}-1)t + \bigg(c_{1}^{2} + c_{2}^{2}+1+2c_{1}c_{2}-2c_{1}-2c_{2} - \mu(C_{1}) - \mu(C_{2}) -2c_{1}c_{2}+r\bigg)t^{2} &= \\
1 + (c_{1}+c_{2}-1)t + \bigg((c_{1}-1)^{2} - \mu(C_{1}) + (c_{2}-1)^{2} - \mu(C_{2})+r-1\bigg)t^{2} &= \\
\mathfrak{P}(C_{1};t) + \mathfrak{P}(C_{2};t) + t-1 + (r-1)t^{2}
\end{align*}
and this completes the proof.
\end{proof}
Let us present some examples that show how to use the above technique in the case of curve arrangements with ordinary quasi-homogeneous singularities. For such arrangements we denote by $n_{i}$ the number of ordinary $i$-fold intersections. 
\begin{example}
Consider the line arrangement $C_{1}$ defined by
$$Q_{1}(x,y,z) = (x-z)(x+z)(y-z)(y+z)(y-x)(y+x)$$
and another curve $C_{2}$ given by
$$Q_{2}(x,y,z) = x^{2}+y^{2}-2z^{2}.$$
Note that $C_{1}$ has $n_{2}=3$ and $n_{3}=4$, and its Poincar\'e polynomial is
$$\mathfrak{P}(C_{1};t) = 1 + 5t + 6t^{2}.$$
The curve $C_{2}$ is just a smooth conic, so its Poincar\'e polynomial is simple, namely
$$\mathfrak{P}(C_{2};t) = 1 + t + t^{2}.$$
Next, we can compute that $r= |C_{1}\cap C_{2}|=4$. Now are in a position to use \eqref{addition}, we have
\begin{multline*}
\mathfrak{P}(C_{1} \cup C_{2};t) = \mathfrak{P}(C_{1};t) + \mathfrak{P}(C_{2};t) + t-1 + (r-1)t^{2} = \\ (1 + 5t + 6t^{2}) + (1 + t + t^{2}) + t-1 + 3t^{2}  = 1 + 7t + 10t^{2} = (1+2t)(1+5t),
\end{multline*}
so this might suggest that $C_{1}\cup C_{2}$ is free, and this is indeed the case by \cite[Theorem 1.3]{Pokora1}.
\end{example}

\begin{example}
Consider the line arrangement $C_{1}$ defined by
$$Q_{1}(x,y,z) = x(x-z)(x+z)(y-z)(y+z)(y-x)(y+x)$$
and another curve $C_{2}$ given by
$$Q_{2}(x,y,z) = x^{2}+y^{2}-2z^{2}.$$
The line arrangement $C_{1}$ consists of $7$ lines and it has $n_{2}=3$ and $n_{3}=6$. We can easily check, using \verb}SINGULAR} \cite{Singular}, that the curve $C_{1}$ is free with exponents $(3,3)$, and hence
$$\mathfrak{P}(C_{1};t) = 1 + 6t + 9t^{2}.$$
If we add $C_{2}$ to $C_{1}$, then the resulting arrangement has degree $9$, and it has the following intersections:
$$n_{2}=5, \quad n_{3}=2, \quad n_{4}=4.$$
We can calculate that $r= |C_{1}\cap C_{2}|=6$, and we use \eqref{addition}, namely
\begin{multline*}
\mathfrak{P}(C_{1} \cup C_{2};t) = \mathfrak{P}(C_{1};t) + \mathfrak{P}(C_{2};t) + t-1 + (r-1)t^{2} = \\ (1 + 6t + 9t^{2}) + (1 + t + t^{2}) + t-1 + 5t^{2}  = 1 + 8t + 15t^{2} = (1+3t)(1+5t).
\end{multline*}
By \cite[Theorem 1.3]{Pokora1}, the curve $C_{1} \cup C_{2}$ is free with exponents $(3,5)$.
\end{example}
\begin{example}[{cf. \cite[Corollary 1.6]{DIPS}}]
Consider the pencil of three lines $C_{1}$ defined by
$$Q_{1}(x,y,z) = x^{3}+y^{3}$$
and the Fermat elliptic curve $C_{2}$ given by
$$Q_{2}(x,y,z) = x^{3}+y^{3}+z^{3}.$$
It is known that every pencil of $d$ lines is free with exponents $(0, d-1)$, so in our situation we have
$$\mathfrak{P}(C_{1};t) = 1 + 2t.$$
The curve $C_{2}$ is just a smooth cubic curve, hence its Poincar\'e polynomial has the following form
$$\mathfrak{P}(C_{2};t) = 1 + 2t + 4t^{2}.$$
If we add $C_{2}$ to $C_{1}$, then the resulting arrangement has degree $6$, and it has one ordinary triple point and three singularities of type $A_{5}$ coming from three inflectional tangents, so we have $r= |C_{1}\cap C_{2} | = 3$. Now we use \eqref{addition} to compute the Poincar\'e polynomial of $C_{1} \cup C_{2}$, namely
\begin{multline*}
\mathfrak{P}(C_{1} \cup C_{2};t) = \mathfrak{P}(C_{1};t) + \mathfrak{P}(C_{2};t) + t-1 + (r-1)t^{2} = \\ (1 + 2t) + (1 + 2t + 4t^{2}) + t-1 + 2t^{2}  = 1 + 5t + 6t^{2} = (1+2t)(1+3t),
\end{multline*}
and by \cite[Corollary 1.6]{DIPS} the curve $C_{1} \cup C_{2}$ is free with exponents $(2,3)$.
\end{example}
\begin{remark}
The above examples show us the main application of the addition for Poincar\'e polynomials, namely we can use it to construct new examples of free curves by adding curves to free ones. In the above examples, we started with free line arrangements, after adding a smooth curve we easily computed Poincar\'e polynomials and checked the splitting over the rationals, which suggested us to directly verify the freeness property.
\end{remark}
Now we pass to the Betti polynomials. We have the following general result.
\begin{theorem}
In the setting as above, 
\begin{equation}
\label{key}
B_{M(C_{1} \cup C_{2})}(t) = B_{M(C_{1})}(t) + B_{M(C_{2})}(t) +t-1+(r-1)t^{2}.
\end{equation}
\end{theorem}
\begin{proof}
We start with the left-hand side, one has
\begin{align*}
\
B_{M(C_{1} \cup C_{2})}(t) &=  \\1 + (e_{1}+e_{2}-1)t +  \bigg((c_{1}+c_{2}-1)^{2} - \mu(C_{1}\cup C_{2}) - (c_{1}+c_{2})+(e_{1}+e_{2})\bigg)t^{2} & =  \\
1 + (e_{1}+e_{2}-1)t + \bigg((c_{1}+c_{2}-1)^{2} - \mu(C_{1}) - \mu(C_{2}) -2c_{1}c_{2}+r - (c_{1}+c_{2})+(e_{1}+e_{2})\bigg)t^{2} & = \\
1+(e_{1}+e_{2}-1)t + \bigg( (c_{1}-1)^{2} + (c_{2}-1)^{2} -\mu(C_{1}) - \mu(C_{2})- c_{1} - c_{2} + e_{1} + e_{2} +r-1\bigg)t^{2} & = \\
B_{M(C_{1})}(t) + B_{M(C_{2})}(t) + (t-1) + (r-1)t^{2}
\end{align*}
and this completes the proof.
\end{proof}
\begin{remark}
If we evaluate the identity \eqref{key} by taking $t=-1$, we get a relation involving the corresponding Euler numbers, namely
$$e(M(C_{1} \cup C_{2})) = e(M(C_{1})) + e(M(C_{2}))+r-3.$$
\end{remark}
\section*{Acknowledgments}
I am indebted to Alex Dimca for all discussions regarding the content of this paper. I would also like to thank an anonymous referee for comments that allowed me to improve the paper.

Piotr Pokora is supported by the National Science Centre (Poland) Sonata Bis Grant  \textbf{2023/50/E/ST1/00025}. For the purpose of Open Access, the author has applied a CC-BY public copyright licence to any Author Accepted Manuscript (AAM) version arising from this submission.

\vskip 0.5 cm
%***************************************************************************** % Addresses
\bigskip

Piotr Pokora,
Department of Mathematics,
University of the National Education Commission Krakow,
Podchor\c a\.zych 2,
PL-30-084 Krak\'ow, Poland. \\
\nopagebreak
\textit{E-mail address:} \texttt{piotr.pokora@uken.krakow.pl}

\begin{thebibliography}{000}
\bibitem{Singular}
W.~Decker, G.-M. Greuel, G.~Pfister, and H.~Sch\"onemann,
\newblock {\sc Singular} {4-1-1} --- {A} computer algebra system for polynomial computations. \newblock \url{http://www.singular.uni-kl.de}, 2018.
\bibitem{Dimca}
A. Dimca,  \textit{Hyperplane arrangements. An introduction}. Universitext. Cham: Springer (ISBN 978-3-319-56220-9/pbk; 978-3-319-56221-6/ebook). xii, 200 p. (2017).
 \bibitem{POG}
A. Dimca, On plus-one generated curves arising from free curves. \textit{Bull. Math. Sci.} \textbf{14(3)}: Art. Id. 2450007 (2024).
\bibitem{Dimca1}
A. Dimca, Some remarks on plane curves related to freeness. \textbf{arXiv:2501.01807}.
\bibitem{DIPS}
A. Dimca, G. Ilardi, P. Pokora, G. Sticlaru, Construction of free curves by adding lines to a given curve. \textit{Result. Math.} \textit{79(1)}: Paper No. 11, 31 p. (2024).
\bibitem{duP}
A. Du Plessis and C. T. C. Wall, Application of the theory of the discriminant to highly singular plane curves. \textit{Math. Proc. Camb. Philos. Soc.} \textbf{126(2)}: 259 -- 266 (1999).
\bibitem{Pokora1}
P. Pokora, Freeness of arrangements of lines and one conic with ordinary quasi-homogeneous singularities. \textit{Taiwanese J. Math.} \textbf{29(6)}: 1651 -- 1666 (2025).
\bibitem{Pokora}
P. Pokora, On Poincar\'e polynomials for plane curves with quasi-homogeneous singularities. \textit{Bull. Lond. Math. Soc.} \textbf{57(8)}: 2549 -- 2560 (2025).
\bibitem{KS1}
K. Saito, Theory of logarithmic differential forms and logarithmic vector fields. \textit{J. Fac. Sci. Univ. Tokyo Sect. IA Math} \textbf{27(2)}: 265 -- 291 (1980).

\bibitem{Wall}
C. T. C. Wall, \textit{Singular points of plane curves}. 
London Mathematical Society Student Texts 63. Cambridge: Cambridge University Press (ISBN 0-521-83904-1/hbk; 0-521-54774-1/pbk). xi, 370 p. (2004).


\end{thebibliography}
\end{document}